\documentclass[12pt]{amsart}
\usepackage{amssymb,amsmath}
\usepackage{geometry}    
\usepackage{mathrsfs}

\usepackage{graphicx}

\usepackage{vmargin}

\usepackage{tikz}  

\setmargrb{1in}{1in}{1in}{1in}

\newtheorem{theorem}{Theorem}[section]
\newtheorem{lemma}[theorem]{Lemma}
\newtheorem{proposition}{Proposition}

\theoremstyle{definition}

\newtheorem{problem}{Problem}

\usepackage{yfonts}

\usepackage{xcolor}

\begin{document}
	
	\title[On an inhomogeneous uniform Littlewood type problem]{On an inhomogeneous uniform Littlewood type problem}
	
	\author{Johannes Schleischitz} 
	
	\thanks{ School of Computer, data and mathematical sciences (CDMS),
    Western Sydney University, Australia  \\
		J.schleischitz@westernsydney.edu.au}

\begin{abstract}
      We show that a fully inhomogeneous uniform Littlewood type problem has a negative answer
      and the counterexamples form a dense $G_{\delta}$ set.
      This extends the author's recent analogous result
      for the homogeneous case.
      The main difficulty in the general setting is the semi-homogeneous case where one linear form is
      homogeneous and the other inhomogeneous with irrational shift. We further address the higher dimensional analogue.
      % Moreover we significantly improve the bound for the homogeneous case by a small twist of the method
      % by the author from previous work.
\end{abstract}

\maketitle

{\footnotesize{

		{\em Keywords}: inhomogeneous approximaton, uniform approximation, Littlewood conjecture \\
		Math Subject Classification 2020: 11H46, 11J20}}

\section{ An inhomogeneous uniform Littlewood conjecture }  \label{intro}

    In~\cite{bfk} the ``Uniform Littlewood Conjecture'' (ULC) was
    introduced. In the special case of simultaneous approximation
    to two variables, it claims that for any real numbers $\xi, \zeta$
    we have
    \[
    \lim_{Q\to\infty} Q \min_{q\in \mathbb{Z}, 0<q\le Q} \Vert q\xi\Vert\cdot \Vert q \zeta\Vert = 0.
    \]
    Here $\Vert.\Vert$ denotes the distance to the nearest integer.
    This was proved to be false by the author~\cite{s}, in fact the 
    counterexamples were shown to contain a dense $G_{\delta}$ set.
    In this paper we study a fully inhomogeneous version 
    of ULC.
    We ask if we can extend the aforementioned 
    topological result from~\cite{s}.
     
    \begin{problem}  \label{p1}
        Given any $\theta_1, \theta_2\in\mathbb{R}$, does the set 
        \[
        \Theta(\theta_1,\theta_2):= \left\{ (\xi, \zeta)\in\mathbb{R}^2: \; \limsup_{Q\to\infty} Q \min_{q\in \mathbb{Z}, 0<q\le Q} \Vert q\xi-\theta_1\Vert\cdot \Vert q \zeta-\theta_2\Vert > 0 \right\}
        \]
        always contain a dense $G_{\delta}$ set?
    \end{problem}

     We recall
  that the dense $G_{\delta}$ property implies many interesting
  other properties such as a full sumset and full packing dimension,
  as detailed in~\cite{s}, hence Problem~\ref{p1} appears of interest in several aspects.

   We should point out that apart from the fully homogeneous 
   case studied in~\cite{bfk, s}, examples of $(\xi, \zeta)\in\Theta(\theta_1,\theta_2)$
   can easily be found. Indeed if both $\theta_i\notin \mathbb{Z}$ then
   we may just take
  $\xi, \zeta$ arbitrary integers. 
  More generally, if at least one $\theta_i\notin\mathbb{Z}$,
  by considering for $\xi$ and/or $\zeta$ certain rational numbers, 
  the set of counterexamples can be readily shown 
  to be uncountable and dense. Even more generally it is not hard to prove the following.

   \begin{theorem}  \label{einfach}
       Assume $\theta_1\notin\mathbb{Z}$ or $\theta_2\notin\mathbb{Z}$. Then there
       is a dense set of $(\xi,\zeta)\in\mathbb{R}^2$ of Hausdorff dimension at least $1$ satisfying
       \begin{equation}  \label{eq:infalsch}
            \liminf_{Q\to\infty} Q \min_{q\in \mathbb{Z}, 0<q\le Q} \Vert q\xi-\theta_1\Vert\cdot \Vert q \zeta-\theta_2\Vert = 
            \liminf_{q\to\infty} q \Vert q\xi-\theta_1\Vert\cdot \Vert q \zeta-\theta_2\Vert 
            > 0.
       \end{equation}
   \end{theorem}

   \begin{proof}
       We may assume $\theta_1\notin\mathbb{Z}$. 
       Let $T\subseteq \mathbb{Q}$ be the set of rationals with denominators not divisible by any prime dividing the denominator of $\theta_1$ if it is rational (we may take $T=\mathbb{Q}$ if $\theta_1\notin\mathbb{Q}$). In any case $T$ is clearly dense in $\mathbb{R}$.
       Note that then for $\xi\in T$
       we have
       $\Vert q\xi-\theta_1\Vert \gg_{\xi} 1$ for all $q$.
       Hence
       any $(\xi,\zeta)\in T\times Bad(\theta_2)$ 
       satisfies \eqref{eq:infalsch}, where
       $Bad(\theta_2)$ is the set of $\theta_2$-badly approximable real numbers, i.e. satisfying
       \[
       \Vert q\zeta-\theta_2\Vert \gg q^{-1}, \qquad q\ge 1.
       \]
       It is well-known $Bad(\theta_2)$ has Hausdorff dimension $1$ in fact it is a winning set~\cite{vic, ma}, hence the claim follows.
   \end{proof}

  Note that even the lower limit is used in Theorem~\ref{einfach}.
  For the uniform sets $\Theta(\theta_1,\theta_2)$, there are also satisfactory metrical results available
  in~\cite{kimkim} for the non-homogeneous case.
  
  On the other hand, neither the easy, rather pathological examples
 from the above proof, nor from~\cite{kimkim}, provide us with a dense $G_{\delta}$ set in Problem~\ref{p1}. In fact
  the set in Theorem~\ref{einfach}
  using the lower limit is a meager set. Indeed, any inhomogeneous Liouville vector with respect to $\theta_1,\theta_2$ induces lower limit $0$, and they form a dense $G_{\delta}$ set; similarly, metrical information is incapable of implying the dense $G_{\delta}$ property, again in view of the Hausdorff dimension $0$ but dense $G_{\delta}$ set of Liouville vectors, whose complement is thus metrically as large as one could hope for, yet meager.
  %  Thus, it is not clear from the argument above if
  % the counterexamples to the limsup problem form a dense $G_{\delta}$ set as in the homogeneous case~\cite{s}.
  % So since
  % the homogeneous case has been done 
  % we may assume
  % \[
  % 0=\theta_1<\theta_2<1.
  % \]
  Our main result verifies Problem~\ref{p1}.
  \begin{theorem}  \label{haupt}
      Problem~\ref{p1} has a positive answer.
  \end{theorem}

  % We may assume $\theta_1=0$ and $\theta=\theta_2\notin\mathbb{Z}$ by above argument.

   We can provide a positive constant uniform in $\theta_1,\theta_2$ unless
   in the case where one $\theta_j$ is an integer and the other $\theta_{3-j}$ either lies in $\mathbb{Q}\setminus \mathbb{Z}$ or is close to an integer (i.e. if either both $\theta_j\notin\mathbb{Z}$ or 
   some $\theta_j\in\mathbb{Z}$ and the other
   $\theta_{3-j}\notin\mathbb{Q}$ satisfies
   $\Vert \theta_{3-j}\Vert\ge \delta>0$).
   The main difficulty in proving Theorem~\ref{haupt}
   in full generality is to get a {\em uniform} positive constant for the limsup expression on a dense set. For example, in the construction
   from the proof of Theorem~\ref{einfach}, 
   the constant depends on $\xi\in T\subseteq \mathbb{Q}$ and necessarily tends to $0$
   as the denominator of $\xi$ blows up.
   However, it appears indeed a uniform lower bound on a dense set is needed to conclude the dense $G_{\delta}$ property.
    The most delicate instance is the ``semi-homogeneous'' case $\theta_j\in\mathbb{Z}, \theta_{3-j}\in\mathbb{R}\setminus \mathbb{Q}$, treated in \S~\ref{casec} below.

    In \S~\ref{high} we will address (a special instance of) Problem~\ref{p1}
     in higher dimension, however our result for this case is rather unsatisfactory.

     We want to close the introduction with 
     an open problem on extending
     Theorem~\ref{haupt} in another direction.

     \begin{problem}
         With $\Theta(\theta_1,\theta_2)$ defined in Problem~\ref{p1}, does the set
         \[
         \bigcap_{(\theta_1,\theta_2)\in\mathbb{R}^2}  \Theta(\theta_1,\theta_2)
         \]
          contain a dense $G_{\delta}$ set (or even is it non-empty)?
     \end{problem}

    Presumably the answer is negative.
     The one-dimensional analogue would be asking if the intersection
     of $Bad(\theta)^c$ sets over all real $\theta$ is 
     dense $G_{\delta}$, or equivalently if the union over all $Bad(\theta)$ is contained in a nowhere dense $F_{\sigma}$ set.
     Note that clearly the union is not countable.

 % Separately we want to use the opportunity to strengthen the numerical value for the homogeneous case from [cite ich ULC Theorem 1.1].

 %   \begin{theorem} \label{th2}
 %       We have
 %         \[
 %       \sup_{(\xi,\zeta)\in\mathbb{R}^2 }\;\limsup_{Q\to\infty} Q \min_{q\in \mathbb{Z},0<q\le Q} \Vert q\xi\Vert\cdot \Vert q \zeta\Vert >  0.01630\ldots> \frac{1}{62}.
 %       \]
 %       and the according set is again dense $G_{\delta}$.
 %   \end{theorem}

 %   This improves $1/188$ from [cite ich ULC] roughly by a factor $3$. Recall that the trivial upper bound is $1/2$.

   \section{Proof Theorem~\ref{haupt}}

\subsection{Splitting into cases}
     We split into four exhaustive cases

     \begin{itemize}
         \item[(a)] $\theta_1, \theta_2\in\mathbb{Z}$
         \item[(b)] $\theta_1\in\mathbb{Z}, \theta_2\in\mathbb{Q}\setminus \mathbb{Z}$
         \item[(c)] $\theta_1\in\mathbb{Z}, \theta_2\in\mathbb{R}\setminus \mathbb{Q}$
         \item[(d)] $\theta_1, \theta_2\in\mathbb{R}\setminus \mathbb{Z}$
     \end{itemize}

     It can easily be checked that indeed this covers all cases.

     Case (a) was done in~\cite{s}.

     For case (b) note that if $\theta_2=m/n$
     then for any $\zeta$ and any integer $q$ we have
     \[
     \Vert q\frac{\zeta}{n}\Vert \ge \frac{1}{n} \Vert q\zeta\Vert,
     \]
     since $|q\zeta/n-p|=(1/n)|q\zeta-np|\ge (1/n)\Vert q\zeta\Vert$
     for any $p\in\mathbb{Z}$.
     Hence any $(\xi, \zeta)\in \Theta(0,0)$ from case (a) gives rise
     to $(\xi,\zeta/n)\in \Theta(\theta_1,\theta_2)$ from case (b). Thus the set $\Theta(\theta_1,\theta_2)$ clearly again contains some dense $G_{\delta}$ set as a dense diffeomorphic image of such a set.

    Case (d) is contained in the following stronger
    claim related to Theorem~\ref{einfach}.
 % We close by stating an optimal inhomogeneous Theorem about the classical ULC as in [cite BFK].

       \begin{theorem}  \label{thmI}
        Let $k\ge 2$ an integer and $\theta_1,\ldots,\theta_k\in\mathbb{R}$. Further
        let $\Phi: \mathbb{N}\to (0,\infty)$ any function
        with $\lim_{t\to\infty} \Phi(t)=\infty$. Then
        \begin{itemize}
            \item[(i)] If $\theta_i\notin\mathbb{Z}$ for $1\le i\le k$ then the set
            \[
            \left\{ (\xi_1,\ldots,\xi_k)\in\mathbb{R}^k:\; \limsup_{Q\to\infty} \Phi(Q) \min_{q\in \mathbb{Z},0<q\le Q} \prod_{j=1}^{k} \Vert q\xi_j-\theta_j\Vert = \infty \right\}
            \]
            contains a dense $G_{\delta}$ subset of $\mathbb{R}^k$.
            \item[(ii)] If $\sharp\{ i: \theta_i\in \mathbb{Z}\}\le 2$ then
            \[
            \left\{ (\xi_1,\ldots,\xi_k)\in\mathbb{R}^k:\;\limsup_{Q\to\infty} Q \Phi(Q) \min_{q\in \mathbb{Z},0<q\le Q} \prod_{j=1}^{k} \Vert q\xi_j-\theta_j\Vert = \infty \right\}
        \]
         contains a dense $G_{\delta}$ subset of $\mathbb{R}^k$.
        \end{itemize}

        % If
        % $\theta_i$ are not both integers then
        % the set of counterexamples $(\xi,\zeta)\in\mathbb{R}^2$ to inhomogeneous ULC satisfying 
        % \[
        % \limsup_{Q\to\infty} Q \min_{q\in \mathbb{Z},0<q\le Q} \Vert q\xi-\theta_1\Vert \cdot \Vert q \zeta-\theta_2\Vert > 0
        % \]
        % contain a dense $G_{\delta}$ subset of $\mathbb{R}^2$. If $\theta_1, \theta_2$ are both not an integer then the same is true for
        % \[
        % \limsup_{Q\to\infty} Q^{\delta} \min_{q\in \mathbb{Z},0<q\le Q} \Vert q\xi-\theta_1\Vert \cdot \Vert q \zeta-\theta_2\Vert = \infty.
        % \]
    \end{theorem}

    Note that there is no $Q$ factor in (i), thus
    it directly implies case (d) when taking
    $\Phi(t)=t$ and $k=2$.
    In (i) the term $\Phi(Q)$ cannot be dropped (even when only imposing positive limsup)
    since by Kronecker's Theorem 
    and the trivial estimate $\Vert q\xi_j-\theta_j\Vert\le 1/2$
    this forces all $\theta_j\in\mathbb{Q}$ so we get at most a countable set, hence no dense $G_{\delta}$ set. 
    In (ii) however it is unclear if it can 
    be dropped, see \S~\ref{high} below for partial results in this direction. Moreover we do not know if the assumption on the cardinality of integral $\theta_j$ in (ii) is necessary, even in the homogeneous case, i.e. all $\theta_j\in\mathbb{Z}$.
    We refer the reader to~\cite{kimkim} for related 
    Hausdorff dimension
    results.
    
    The proof of Theorem~\ref{thmI} is surprisingly easy, upon using in (ii) the homogeneous case (a) result from~\cite{s}.

    \subsection{Proof Theorem~\ref{thmI}}
    %Let us first prove the second claim.
    Similar to Theorem~\ref{einfach},
    it is easy to see that when $\theta_j\notin \mathbb{Z}$ for all $j$ then we can
    choose $\xi_j$ in a dense set of rational numbers (all rationals if $\theta_j$ is irrational)
    so that $\Vert q\xi_j-\theta_j\Vert \gg_{\xi_j} 1$ for all $q$ and $1\le j\le k$.

    Now for the first claim note that this implies that the set
    \begin{equation} \label{eq:shawa}
     \left\{ (\xi_1,\ldots,\xi_k)\in\mathbb{R}^k:\; \limsup_{Q\to\infty} \min_{q\in \mathbb{Z},0<q\le Q} \prod_{j=1}^{k} \Vert q\xi_j-\theta_j\Vert > 0 \right\}
    \end{equation}
    is dense in $\mathbb{R}^k$. However our set in (i) is larger since $\Phi\to\infty$, so it is also dense.
     The $G_{\delta}$ property follows similarly as in~\cite{s}: Consider the (by continuity) closed sets
    \[
    \Lambda_Q:=\{ (\xi_1,\ldots,\xi_k)\in\mathbb{R}^k: \; \min_{q\in \mathbb{Z},0<q\le Q} \prod_{j=1}^{k} \Vert q\xi_j-\theta_j\Vert \le \frac{1}{\Phi(Q)} \}.
    \]
    Their countable intersection $Y_{Q_0}= \cap_{Q\ge Q_0} \Lambda_Q$ 
    over integers $Q\ge Q_0$ is again closed,
    hence $Y:=\cup_{t\ge 1} Y_t$ with union over integers $t$ is $F_{\sigma}$, however its $G_{\delta}$ complement $Y^c$ is essentially the set of (i). (It gives limsup $\ge 1$, however
    by applying the argument to $\sqrt{\Phi}$ which also tends to infinity, one easily checks
    that we may write $\infty$ instead).
     This proves (i).

    For (ii) note that 
    a fortiori we may assume that for exactly two indices $j$ we have $\theta_j\in\mathbb{Z}$, we may further relabel so that $j=1,2$. Now note that 
    just for $k=2$ and $\theta_1,\theta_2\in\mathbb{Z}$
    the claim was proved in~\cite{s}, call $\Theta\subseteq \mathbb{R}^2$
    the according dense $G_{\delta}$ set.
    Thus again by the argument from
    the proof of Theorem~\ref{einfach}
    we get $\Sigma:=\Theta\times T^{k-2}$
    is contained in the set \eqref{eq:shawa}, where $T\subseteq \mathbb{Q}^{k-2}$
    is dense in $\mathbb{R}^{k-2}$. Clearly then $\Sigma$ is dense
    in $\mathbb{R}^k$, thus the set in (ii) is also dense. The $G_{\delta}$ property follows analogously to (i), just altering the right hand side in $\Lambda_Q$ to $1/(Q\Phi(Q))$. Theorem~\ref{thmI} is proved.

     The last open case (c) is the most challenging.
     The proof below employs some ideas
     of the homogeneous case in~\cite{s}, however also several new elements.

      \subsection{Proof main case (c)}  \label{casec}

      \subsubsection{Preliminaries}
      For simplicity we can assume
      \[
      \theta_1=0, \qquad \theta_2=: \theta.
      \]
      We recall an easy observation~\cite[Proposition~3]{s}.

    \begin{proposition}[Schleischitz] \label{pro}
    Let $c,d$ be positive real numbers with $cd<1$.
        Let $\xi$ be a real number and assume for some integers $p$ and $q>0$ we have
        \[
        |q\xi-p| < cq^{-2}.
        \]
        Then for any integers $u,v$ so that $u/v\ne p/q$ and 
        $0<v<Q=dq^2$, we have
        \[
        |v\xi-u| \ge (1-cd)\sqrt{d} Q^{-1/2}.
        \]
    \end{proposition}

    % \begin{proof}
    %     Consider the modulus of the 
    %     determinant of the matrix 
    %     \[
    %     M:= \begin{pmatrix}
    %         p & q \\
    %         u & v
    %     \end{pmatrix}.
    %     \]
    %     It is at least $|\det M|\ge 1$ 
    %     by $p/q\ne u/v$, on the other hand subtracting 
    %     $\xi$ times the second column from the first column 
    %     and expanding, we get
    %     it is at most $|\det M|\le q|v\xi-u| + Q cq^{-t} \le q|v\xi-u| +cd$. Combining these bounds we get
    %     \[
    %     |v\xi-u|\ge (1-cd)q^{-t}= (1-cd)d^{1/t} Q^{-1/t} .
    %     \]
    % \end{proof}

     We can choose $d>1$ close to $1$ in the proof of case (c).
     We also need 

     \begin{proposition}  \label{nothardp}
         For any irrational real number $\theta$ there exist infinitely
         many coprime pairs of integers $q,R$ with 
         \[
         (1/2-1/\sqrt{5}) q^{-1} < |q\theta-R|< 2\cdot (1+1/\sqrt{5}) q^{-1}.
         \]
     \end{proposition}

     \begin{proof}
         It is well-known that 
         \[
         |r\theta-s|< \frac{1}{\sqrt{5}} r^{-1}
         \]
         has infinitely many solutions in integers $r>0, s$ where $s/r$ is a convergent to $\theta$.
         Moreover for the preceding convergent $v/u$ with $0<u<r$ 
         it is well-known that we have
         \[
         r^{-1}>|u\theta-v|> \frac{1}{2} r^{-1}.
         \]
         Hence by triangle inequality and since $r<u+r<2r$, we get
         \[
         (1/2-1/\sqrt{5})(u+r)^{-1}<(1/2-1/\sqrt{5})r^{-1}<| (u+r)\theta-(s+v)\vert
         \]
         and on the other hand
         \[
         | (u+r)\theta-(s+v)\vert < (1+1/\sqrt{5})r^{-1}< 2\cdot (1+1/\sqrt{5})(u+r)^{-1}.
         \]
         Moreover $(u+r,v+s)=1$ since the determinant of the matrix
         with rows $(u+r,v+s)$ and $(u,v)$ is $\pm 1$ as for the matrix
         with rows $(r,s)$ and $(u,v)$. Hence 
         \[
         q=u+r, \qquad R=s+v
         \]
         satisfy all claimed properties.
     \end{proof}

     It may be possible/easy to improve the constants in Proposition~\ref{nothardp}, it is not of major relevance for us.
    
    \subsubsection{Existence proof}
    % Let $\tilde{q}_0>0$ be a good common denominator for $\theta$ more precisely let $R/\tilde{q}_0$ for some integer $R$ be a good approximation to $\theta$, to be made precise later. 
    %Let $N$ be large real number. 
    % with
    % \[
    % \max\{ 1,\frac{2}{\delta}\} < K < N/3.
    % \]
    We first provide some real vector $(\xi,\zeta)$ as in the claim,
    the dense $G_{\delta}$ property will follow easily from our
    construction.
    Let $d>1$ for now arbitrary but fixed.
    For the first step let $I_0, J_0$ be arbitrary non-empty closed subintervals of $(0,1)$. Then assume for the moment that for integers $\tilde{q}_0>0, \tilde{p}_0, \tilde{r}_0, r_0, s_0$ with $(\tilde{q}_0, \tilde{r}_0)=(s_0,r_0)=1$ and $\tilde{p}_0/\tilde{q}_0\in I_0$ and $\tilde{r}_0/\tilde{q}_0\in J_0$ and any real numbers $\xi\in I_1, \zeta\in J_1$ with subintervals $I_1\subseteq I_0, J_1\subseteq J_0$ to be constructed
    we have the following conditions: Firstly
  \begin{equation} \label{eq:2null}  \tag{C1}
      |\zeta-\tilde{r}_0/\tilde{q}_0|< (\tilde{q}_0)^{-N-1}
    \end{equation}
    for very large $N>2$. Secondly for given
     \begin{equation} \label{eq:tatsache}
     0<\alpha<\beta<1/d<1
     \end{equation}
     small enough 
    and for $0<c_1<c_2$ absolute
    constants (all uniform in all steps below) 
    to be chosen later (see \eqref{eq:hier} below), we have
    \begin{equation} \label{eq:genauer}   \tag{C2}
        c_1 \tilde{q}_0<|s_0|< c_2\tilde{q}_0,
    \end{equation}
     and thirdly
     \begin{equation} \label{eq:11null}  \tag{C3}
      \alpha \tilde{q}_0^{-3}<|\xi-\tilde{p}_0/\tilde{q}_0|< \beta \tilde{q}_0^{-3}, \qquad 
    \alpha \tilde{q}_0^{-3}<|\zeta-(r_0+\theta)/s_0|< \beta \tilde{q}_0^{-3},
     \end{equation}
    which by \eqref{eq:genauer} implies
    \begin{equation} \label{eq:Gg}
    \alpha \tilde{q}_0^{-2}<|\tilde{q}_0\xi-\tilde{p}_0|<\beta \tilde{q}_0^{-2}, \qquad 
    c_1 \alpha \tilde{q}_0^{-2}<|s_0\zeta-r_0-\theta|<c_2 \beta \tilde{q}_0^{-2}.
    \end{equation}
    %{\color{red} suffices exponent reduction by $1$??} 
    Let
    % \[
    % \gamma=\max\{ q_0|q_0\xi-p_0-\theta_1|, 
    % s_0|s_0\zeta-r_0-\theta_2|  \} \le \beta
    % \]
    % and %for some parameter $\tau>1$
    \[
    Q=Q_0= d \tilde{q}_0^{2}.
    \]
    %for small $\epsilon>0$.
    We claim that
    for $(u,v)\in\mathbb{Z}^2$ with $|u|\le Q$ for small enough 
    $\sigma>0$ we can have 
    \begin{equation} \label{eq:assumethis}
    |u\zeta-v-\theta|<\sigma Q^{-1/2}
    \end{equation}
    only if
    \begin{equation}  \label{eq:1010}
    (u,v)=(s_0,r_0)+t(\tilde{q}_0,\tilde{p}_0), \qquad t\in\mathbb{Z}.
    \end{equation}
    Assuming \eqref{eq:assumethis}, indeed on one hand by triangle inequality and \eqref{eq:Gg} we get
    \[
     (\sigma+o(1))Q^{-1/2}>\sigma Q^{-1/2} + c_2\beta Q^{-1} \ge |(s_0\zeta-r_0-\theta) - (u\zeta-v-\theta)|= 
    |(u-s_0)\zeta-(v-r_0)|.  
    \]
    On the other hand by \eqref{eq:genauer} and $d>1$
    we can assume $|u-s_0|<\tilde{q}_0^2+c_2\tilde{q}_0<d\tilde{q}_0^2=Q$, 
    so by Proposition~\ref{pro} with the $d$ above and $c=\beta$ yields for other pairs, where $(u-s_0,v-r_0)$ is not a multiple of $(\tilde{q}_0,\tilde{p}_0)$, that
    \[
    |(u-s_0)\zeta-(v-r_0)|\ge (1-d\beta)\sqrt{d}Q^{-1/2}.
    \]
    Combining gives
     \[
     \sigma+o(1) \ge (1-d\beta)\sqrt{d} 
     \]
    so we may take
    \[
    \sigma=\frac{(1-d\beta)\sqrt{d}}{2}>0
    \]
    for large enough $\tilde{q}_0$ by \eqref{eq:tatsache}.
    % (Remark: We may replace $d$ by $d=1+o(1)$ which gives the stronger estimate for $\sigma>1-\beta-o(1)$ for $\beta<1$.)

    Similarly unless 
    \begin{equation}  \label{eq:2020}
    (u,v_1)=t(\tilde{q}_0,\tilde{p}_0), \qquad t\in\mathbb{Z},
    \end{equation}
    by Proposition~\ref{pro} with present $d$ and $c=\beta$ we have
    \[
    \vert u\xi-v_1|\ge |\tilde{q}_0\xi-\tilde{p}_0|\ge (1-d\beta) \sqrt{d}Q^{-1/2}.
    \]
    
    % which by $K<N$ means $(u-q_0,v-p_0)$ must be a multiple of $(\tilde{q}_0,\tilde{p}_0)$ if $\beta$ is small enough, using $\tilde{q}_0\asymp q_0$.

    % Very similarly we must have $(u,v)=(s_0,r_0)+t(\tilde{q}_0,\tilde{r}_0)$ in order to have 
    %   \[
    % |u\zeta-v-\theta_2|<\frac{1}{2}Q^{-1/K}.
    % \]
    Hence if $u, v_1, v_2$ are not inducing either of these two forms 
    \eqref{eq:1010}, \eqref{eq:2020},
    then
    \[
    Q|u\xi-v_1|\cdot |u\zeta-v_2-\theta|
    \ge (1-d\beta) \sqrt{d} \sigma \ge \frac{(1-d\beta)^2 d}{2}>0.
    \]
    We check the remaining two cases \eqref{eq:1010}, \eqref{eq:2020} separately, starting with the latter.
    % Indeed
    % \[
    % Q^{\delta}|u\xi-v_1|\cdot |u\zeta-v_2-\theta_2| \to \infty
    % \]
    % follows.

    Case 1: $(u,v_1)=t(\tilde{q}_0,\tilde{p}_0)$ for $1\le t\le Q/\tilde{q}_0=d\tilde{q}_0$. We readily check by \eqref{eq:Gg} that
    \[
    \vert u\xi - v_1|\ge t |\tilde{q}_0 \xi - \tilde{p}_0|\ge |\tilde{q}_0 \xi - \tilde{p}_0| \ge \alpha d Q^{-1}.
    \]
    On the other hand by \eqref{eq:2null} we have
    $\Vert u \zeta\Vert=t\Vert\tilde{q}_0 \zeta\Vert\le \tilde{q}_0^{-N+1}$ is very small so since $\theta\notin \mathbb{Z}$ we get for any $v_2\in\mathbb{Z}$ that
    \[
    \vert u\zeta - v_2-\theta|\ge 
    \Vert\theta\Vert - \Vert u \zeta\Vert
    \ge (1-o(1)) \Vert\theta\Vert \gg_{\theta} 1.
    \]
    Thus the product satisfies 
    \[
    Q|u\xi-v_1|\cdot |u\zeta-v_2-\theta|
    \ge  (1-o(1)) \cdot \alpha d \Vert\theta\Vert \gg_{\theta} 1
    \]
    indeed.

    Case 2: $(u,v_2)=(s_0,r_0)+t(\tilde{q}_0,\tilde{p}_0)$. Now 
    by \eqref{eq:2null} again
    \[
    |t\tilde{q}_0 \zeta-t\tilde{p}_0|\le t \tilde{q}_0^{-N}\ll \tilde{q}_0^{1-N}\ll Q^{-(N-1)/2}
    \]
     is very small since $N$ is large. Hence
     \[
     |u\zeta-v_2-\theta| = |(s_0 \zeta-r_0-\theta)+(t\tilde{q}_0 \zeta-t\tilde{p}_0)|
     =|s_0 \zeta-r_0-\theta| + O(Q^{-(N-1)/2})
     \]
     which is $\ge (c_1\alpha-o(1))\cdot \tilde{q}_0^{-2}= (c_1\alpha-o(1))d\cdot Q^{-1}$ by \eqref{eq:Gg}. 

     On the other hand
     \[
     \Vert t\tilde{q}_0\xi\Vert\le |t|\cdot \Vert \tilde{q}_0\xi\Vert
     \le (Q/\tilde{q}_0)\cdot (\beta \tilde{q}_0^{-2})
     =(d\tilde{q}_0)\cdot (\beta \tilde{q}_0^{-2}) = \beta d \tilde{q}_0^{-1}
     \]
     by \eqref{eq:Gg}, so we only need
     \[
            \Vert s_0 \xi\Vert\gg 1,
     \]
     to conclude by triangle inequality $\Vert u\xi\Vert \gg 1$. 
     However, since by \eqref{eq:11null} we have $\xi=\tilde{p}_0/\tilde{q}_0+O(\tilde{q}_0^{-3})$
     we equivalently want 
      \begin{equation}  \label{eq:EQA}   \tag{C4}
            \Vert \frac{s_0 \tilde{p}_0}{\tilde{q}_0}\Vert\gg 1.
     \end{equation}
     Then again the product is $\gg Q^{-1}$. Combining the cases indeed we get
     \begin{equation} \label{eq:obenauf}
     Q_0 \min_{u\le Q_0} \Vert u\xi\Vert \cdot \Vert u\zeta-\theta\Vert \gg_{\theta} 1,
     \end{equation}
     for some implied constant depending only on $\theta$, once $\alpha, \beta, d$ are chosen satisfying \eqref{eq:tatsache}. 

     We need to check we may simultaneously satisfy 
     \eqref{eq:2null}-\eqref{eq:EQA}.
% Hence it suffices to have 
%      \[
%      |(q_0+t\tilde{q}_0)\zeta-v_2-\theta_2| \gg Q^{-\delta+1/K}.
%      \]
%      However $\zeta\approx \tilde{r}_0/\tilde{q}_0$ so we need
%      \[
%      |(q_0+t\tilde{q}_0)(\tilde{r}_0/\tilde{q}_0)-v_2-\theta_2| \gg Q^{-\delta+1/K}.
%      \]
%      The left hand side simplifies to 
%      \[
%      |t\cdot \tilde{r}_0 + \frac{q_0 \tilde{r}_0 }{ \tilde{q}_0 }-v_2-\theta_2|.
%      \]
%      So we need
%      \begin{equation}  \label{eq:AH1}
%      \Vert\frac{q_0 \tilde{r}_0 }{ \tilde{q}_0 }-\theta_2\Vert \gg Q^{-\delta+1/K}.
%      \end{equation}
%      Similarly for the other case 
%      $(u,v)=(s_0,r_0)+t(\tilde{q}_0,\tilde{r}_0)$
%      we need
%      \begin{equation}  \label{eq:AH2}
%      \Vert\frac{s_0 \tilde{p}_0 }{ \tilde{q}_0 }-\theta_1\Vert \gg Q^{-\delta+1/K}.
%      \end{equation}   
     So let us look at the two central conditions \eqref{eq:2null}, \eqref{eq:11null}.
     For the conditions involving $\zeta$ we basically need
     \[
     \alpha \tilde{q}_0^{-3} < |(r_0+\theta)/s_0- \tilde{r}_0/\tilde{q}_0| < \beta \tilde{q}_0^{-3}
     \]
     % and 
     %  \[
     % |(r_0+\theta_2)/s_0- \tilde{r}_0/\tilde{q}_0| \ll Q^{-1/K}
     % \]
     then we can choose a very small interval 
     $J_1\ni \zeta$
     around $\tilde{r}_0/\tilde{q}_0\in J_0$ 
     % and $\tilde{r}_0/\tilde{q}_0$ for $\xi$ resp. $\zeta$ 
     to choose $\tilde{r}_1/\tilde{q}_1\in J_1$ for the next step. 
     We can assume $\tilde{r}_0/\tilde{q}_0$ is not an endpoint of $J_0$,
     hence by shrinking $J_1$ if necessary, we clearly may assume $J_1\subseteq J_0$.
     This is equivalent to
     \[
       %   c_1\alpha\tilde{q}_0^{-1}< 
     \alpha s_0 \tilde{q}_0^{-2} <
     | \tilde{q}_0 \theta - R_0|
     < \beta s_0 \tilde{q}_0^{-2}
     %< c_2 \beta \tilde{q}_0^{-1},
     \]
     % {\color{red} LOWER BOUNDS MISSING !!!! }
     where 
     \begin{equation} \label{eq:AH3}
     R_0=-\tilde{q}_0 r_0 + s_0\tilde{r}_0, 
     \end{equation}
     and by \eqref{eq:genauer} we infer the sufficient sharper condition
     \begin{equation}
         c_2 \alpha \tilde{q}_0^{-1} < | \tilde{q}_0 \theta - R_0| < c_1\beta \tilde{q}_0^{-1}.
     \end{equation}
     Hence via Proposition~\ref{nothardp} it suffices to 
     choose $\alpha, \beta, c_1, c_2$ so that
     \begin{equation}  \label{eq:jahalt}
     \alpha c_2 < \frac{1}{2}- \frac{1}{\sqrt{5}} < 2(1+\frac{1}{\sqrt{5}})< \beta c_1,
     \end{equation}
     in order to find arbitrarily large coprime $\tilde{q}_0, R_0$ satisfying these estimates.
      We may choose $d>1$ arbitrarily close to $1$ so that $\beta<1$
      can also be arbitrarily close to $1$, then for $c_1=3$ and $c_2=4$
      and $\alpha$ small enough \eqref{eq:jahalt} holds, so for instance we may put
      \begin{equation} \label{eq:hier}
      c_1=3, \quad c_2=4, \quad \alpha=\frac{1}{100}, \quad \beta=1-\frac{1}{100}, \quad d=1+\frac{1}{100}.
      \end{equation}
     % Choosing $c_2>2\cdot (1+1/\sqrt{5})/\beta$, since $\theta$ is irrational,
     % for $\alpha$ (and/or $c_1$) small enough then
     % by Proposition~\ref{nothardp}
     % there exist arbitrarily large coprime $\tilde{q}_0, R_0$ satisfying these estimates, we start step 1 with these. 
      Lemma~\ref{lemur0}
     below shows that the remaining integers $\tilde{p}_0, \tilde{r}_0, r_0, s_0$ can be suitably chosen as well
to induce $\tilde{q}_0, R_0$ as above and satisfy all other claims.

         \begin{lemma}  \label{lemur0}
        %Let $\theta>0$ be an irrational real number.
        Let $I,J$ be compact intervals of positive length. Further let $c<1/2$ be a constant.
        Given any large enough (depending on $I$, $J$) coprime positive integers $\tilde{q}$ and $R$, there exist  positive integers $r,s,\tilde{p},\tilde{r}$ 
        with $(\tilde{q},\tilde{p}\tilde{r})=1$
        such that
        \begin{equation}  \label{eq:c1}  \tag{c1}
     R=-\tilde{q} r + s\tilde{r} 
        \end{equation}
        and
          \begin{equation}  \label{eq:c2}  \tag{c2}
        \frac{ \tilde{p} }{ \tilde{q} }\in I, \qquad \frac{  \tilde{r} }{ \tilde{q} }\in J
        \end{equation}
        and 
          \begin{equation}  \label{eq:c3}  \tag{c3} 
          3< \frac{s}{ \tilde{q} }<4
        \end{equation}
        and 
          \begin{equation}  \label{eq:c4}  \tag{c4} 
     \left\Vert\frac{s \tilde{p} }{ \tilde{q} }\right\Vert \ge c.
        \end{equation}
    \end{lemma}

%    First assume $\tilde{q}$ is prime.
%    {\color{red} This gives a proof in the case that $\theta$ has i.o. prime convergent denominators;
%    more generally if $\Vert q\alpha\Vert < cq^{-2}$
%    for primes i.o., follows from Linnik's Conjecture
%    on least prime in an arithmetic progression\\
% https://math.stackexchange.com/questions/4775755/conjecture-similar-to-dirichlets-approximation-theorem-but-with-prime-numbers
%    and true for almost all by Hauke, Kowalski https://people.math.ethz.ch/~kowalski/approximation.pdf  }
  \begin{proof}
   Choose arbitrary $\tilde{p}$ coprime to $\tilde{q}$ so that $\tilde{p}/\tilde{q}\in I$ from \eqref{eq:c2} holds (clearly possible for $\tilde{q}$ large enough,
   see eg Lemma 3 from Ch I in~\cite{cassels}).
   For any $\tilde{r}\in J\tilde{q}\cap \mathbb{Z}$ coprime to $\tilde{q}$ we have \eqref{eq:c2}
   and moreover there is a shifted one-dimensional lattice of $(r,s)\in\mathbb{Z}^2$ satisfying the linear congruence condition \eqref{eq:c1},
    given by $\{(r^0,s^0)+\ell (\tilde{r},\tilde{q}): \ell\in\mathbb{Z}\}$ for some $0< r^0, s^0 < \tilde{r}\tilde{q}$, where
    \[
    s\equiv s^0\equiv R\tilde{r}^{-1}\bmod \tilde{q}.
    \]
     By suitable choice of
    $\ell$ we can assume that $3\tilde{q}<s<4\tilde{q}$, so \eqref{eq:c3} holds. We must show for some choice $\tilde{r}$ also \eqref{eq:c4} holds.
   Note that
    \begin{equation} \label{eq:virus}
    \left\Vert\frac{s \tilde{p} }{ \tilde{q} }\right\Vert= 
    \left\Vert\frac{s^0 \tilde{p} }{ \tilde{q} }\right\Vert
    =\left\Vert\frac{R\tilde{r}^{-1}\cdot \tilde{p} }{ \tilde{q} }\right\Vert.
    \end{equation}
    Consider the interval sets 
    \[
    X=[c\tilde{q}, (1-c)\tilde{q}]\cap \mathbb{Z},
    \qquad Y= J \tilde{q}\cap \mathbb{Z}. 
    \]
    Note both intervals have length $\gg \tilde{q}$.
    Then it follows from Theorem 13 and below formula (10) in~\cite{sh} and standard estimates for Euler's totient function $\varphi$
    that for any primitive residue class $h\bmod \tilde{q}$, if $\tilde{q}$ is large enough, 
    there exists a positive cardinality
    \[
    \frac{ \varphi(\tilde{q}) }{ \tilde{q}^2 } \sharp X\sharp Y + O(\tilde{q}^{1/2+o(1)})
    = \frac{ \varphi(\tilde{q}) }{ \tilde{q}^2 } (1-2c)\tilde{q} \cdot(|J|\tilde{q}) + O(\tilde{q}^{1/2+o(1)}) \gg \varphi(\tilde{q}) + O(\tilde{q}^{1/2+o(1)}) > 0
    \]
    of pairs $x\in X, y\in Y$ with $xy\equiv h\bmod \tilde{q}$. If we 
    choose any such pair for $h=R\tilde{p}\bmod \tilde{q}$, 
    since $\tilde{r}\in Y$ by construction,
    this means  $R\tilde{r}^{-1}\tilde{p}$ lies
    in some residue class among $X$ modulo $\tilde{q}$, thus by \eqref{eq:virus} we have
    \eqref{eq:c4}. The lemma is proved.
    \end{proof}

     By identifying $\tilde{q}$ in the lemma with $\tilde{q}_0$ above and the like for all other variables, it is immediate that \eqref{eq:genauer} and \eqref{eq:EQA} hold. Our argument preceding the lemma further shows
     that for any $\zeta\in J_1$ the according conditions from
     \eqref{eq:2null} and \eqref{eq:11null} hold. Finally,
     taking $I_1:=[\tilde{p}_0/\tilde{q}_0+\alpha \tilde{q}_0^{-3}, \tilde{p}_0/\tilde{q}_0+\beta \tilde{q}_0^{-3}]$,
     the condition on $\xi$ from \eqref{eq:11null}
     is obviously met as any $\xi\in I_1$ satisfies the required estimate.
     We may again assume $I_1\subseteq I_0$ for the same reasons as for $J_0, J_1$. This concludes the construction for the first step.

     The next step is to apply the same argument starting from the smaller intervals $I_1, J_1$. We may repeat the argument
     ad infinitum to obtain sequences of nested compact 
     intervals $I_j, J_j$ inducing well-defined numbers
     \[
     \xi=\bigcap_{j\ge 1} I_j, \qquad \zeta=\bigcap_{j\ge 1} J_j
     \]
     for which indeed
     \begin{equation}  \label{eq:foed}
     \limsup_{Q\to\infty} Q \min_{u\in\mathbb{Z}, 0<u\le Q} \Vert u\xi\Vert \cdot \Vert u\zeta-\theta\Vert \gg_{\theta} 1,
     \end{equation}
     holds with the absolute implied constant as in \eqref{eq:obenauf}.
     % gives
     %  \[
     % | \tilde{q}_0 \theta_2 - R_2|\ll Q^{1/K}
     % \]
     % % {\color{red} LOWER BOUNDS MISSING  }
     % where 
     % \begin{equation}  \label{eq:AH4}
     % R_2=\tilde{q}_0 r_0 - q_0\tilde{r}_0.  
     % \end{equation}
     % By Dirichlet (Kronecker's) Theorem this always has a solution.
     % Hence given $\tilde{p}_0,\tilde{r}_0,\tilde{q}_0$ with $\tilde{p}_0/\tilde{q}_0$
     % and $\tilde{r}_0/\tilde{q}_0$ reduced
     % and $\tilde{q}_0/R_1$ resp. $\tilde{q}_0/R_2$ good approximations for $\theta_1, \theta_2$, we can indeed
     % find $q_0, p_0, r_0\ll \tilde{q}_0$ to satisfy this.
     Thus we have finished the existence part of the proof.

 \subsubsection{Dense $G_{\delta}$ property}
  The density of the set in the theorem is immediate from our construction above since $I_0, J_0$ were arbitrary.
 The argument for the $G_{\delta}$ property 
 is analogous to~\cite{s} starting from closed sets
   \[
    \Lambda_{Q,\theta}:=\{ (\xi, \zeta)\in\mathbb{R}^2: \; \min_{q\in\mathbb{Z}, 0<q\le Q} Q \Vert q\xi\Vert\cdot \Vert q\zeta-\theta\Vert \le y \}, 
    \]
    with $y=y(\theta)>0$ the implied constant from the proof above,
 see also the proof of Theorem~\ref{thmI} above. We omit the details.

\section{A result in higher dimension}  \label{high}

   We want to address the situation of more variables in Problem~\ref{p1}.
   We should point out that in the homogeneous case little is known, even if the according set satisfying
   \begin{equation}  \label{eq:wideopen}
        \Theta_3:= \left\{ (\xi_1, \xi_2, \xi_3)\in\mathbb{R}^3: \; \limsup_{Q\to\infty} Q \min_{q\in \mathbb{Z}, 0<q\le Q} \Vert q\xi_1\Vert\cdot \Vert q \xi_2\Vert \cdot \Vert q \xi_3\Vert> 0 \right\}
        \end{equation}
   is non-empty remains open. In~\cite{bfk} it is conjectured
   that this is false, meaning the authors expect 
  $\Theta_3$ to be empty. This may be a of similar difficulty 
  as the notoriously difficult classical Littlewood problem.
   In view of Theorems~\ref{einfach}, 
   Theorem~\ref{thmI} and their proofs, 
   it appears the problem
   should become easier in an inhomogeneous setting with (certain) 
   irrational $\theta_j$. 
    Indeed recall that Theorem~\ref{thmI}, (i), claims the existence of a dense $G_{\delta}$ set of counterexamples follows in any dimension
    when all $\theta_j$ are irrational. 
    We want to prove here a semi-homogeneous setting with
    exactly one $\theta_1\in\mathbb{Z}$ and all other $\theta_j$ irrational, extending the case for $k=2$ variables from \S~\ref{casec}. 
     In
   higher dimension we require a certain Diophantine 
   restriction for the inhomogeneous terms for our method. Concretely our result, with slight altered notation to highlight the homogeneous variable, reads as follows.

   \begin{theorem}  \label{konsch}
       Let $k\ge 2$ be an integer and assume for some $0<\gamma_1<\gamma_2$ and irrational real numbers $\theta_1,\ldots,\theta_{k-1}$ the system 
       \begin{equation} \label{eq:Diese}
       \gamma_1 q^{-(k-1)} < \max_{1\le j\le k-1} \Vert q\theta_j\Vert < \gamma_2 q^{-(k-1)} 
       \end{equation}
       has infinitely many solutions in integers $q$. Then the set 
       \[
       \left\{ (\xi,\zeta_1,\ldots,\zeta_{k-1})\in\mathbb{R}^k:\;\limsup_{Q\to\infty} Q \min_{q\in \mathbb{Z},0<q\le Q} \Vert q\xi\Vert \cdot  \prod_{j=1}^{k-1} \Vert q\zeta_j-\theta_j\Vert > 0 \right\}
       \]
       contains a dense $G_{\delta}$ set.
   \end{theorem}

   %We will apply this with $\ell=k-1$.
   % While we let $k=3$, most of the proof works and is executed 
   % for any $k\ge 3$. 
   % We strongly believe it remains true for larger $k$, 
   % however there are some technical
   % problems in the proof of Lemma~\ref{lemur00} below when $k>3$.
   % The restriction the prime denominators in \eqref{eq:Diese} is also supposedly not
   % necessary and again artefact of the proof of Lemma~\ref{lemur00}. 
   The condition \eqref{eq:Diese} is metrically 
   non-generic for $k\ge 3$, it defines a set of
   Hausdorff dimension $1$ in $\mathbb{R}^{k-1}$. See for example~\cite{desaxce}, 
   extending a well-known formula of Jarn\'ik~\cite{jarnik} where only the upper estimate in \eqref{eq:Diese} is imposed.
   On the other hand, topologically it is large, indeed it can be shown the set
   itself conatins a dense $G_{\delta}$ subset of $\mathbb{R}^{k-1}$, by a similar short argument 
    as in~\cite[Lemma 8.1]{s}.
    %Hence we obtain
 %   \begin{corollary}
 %       The set
 % \[
 %       \left\{ (\xi,\zeta_1,\ldots,\zeta_{k-1}, \theta_1,\ldots,\theta_{k-1})\in\mathbb{R}^{2k-1}:\;\limsup_{Q\to\infty} Q \min_{q\in \mathbb{Z},0<q\le Q} \Vert q\xi\Vert \cdot  \prod_{j=1}^{k-1} \Vert q\zeta_j-\theta_j\Vert > 0 \right\}
 %       \]
 %       contains a dense $G_{\delta}$ set.     
 %   \end{corollary}
 %   %
   % set as it contains
   % the Liouville vectors (for which the exponent $k-1$ can be replaced
   % by an arbitrarily large number).  
   We believe the claim of Theorem~\ref{konsch} remains 
   true without any condition
   like \eqref{eq:Diese} on the $\theta_j$ besides (possibly) irrationality.
   
   We should further note that the setting of exactly two 
   homogeneous linear forms is of interest as well, and presumably more complicated. 
   If positive upper limit is induced by a superset of
   a dense $G_{\delta}$ set, this 
   would extend
   the homogeneous result for two variables in~\cite{s}. 
   For three or more 
   homogeneous variables this leads to first study \eqref{eq:wideopen} which is wide open. We have no contribution to these cases.

     %  Identifying $\kappa=1-2c$ the claim of the lemma follows.

    % For $k=2$ the conditions \eqref{eq:c10}-\eqref{eq:c40} essentially simplify to the situation of the unconditional 
    % Lemma~\ref{lemur0}. The first three
    % conditions can be dealt with very similarly as in this special case.
    % However,
    % the last equidistribution type condition \eqref{eq:c40}, which 
    % while from a heuristic point of view must be true,
    % causes technical problems in the proof 
    % when $k>2$. Even $k=3$ turns out problematic. 
    % We invite readers to find a rigorous proof
    % to make Conjecture~\ref{konsch} a non-conditional result.
   
    We explain how to modify the proof of case (c) of Theorem~\ref{haupt} 
    in \S~\ref{casec} to prove Theorem~\ref{konsch}.  
     We need the following generalization of Lemma~\ref{lemur0}.

    \begin{lemma}  \label{lemur00}
        %Let $\theta>0$ be an irrational real number.
        %Let $k=3$.
        % Assume for an integer $k\ge 2$ the following holds for 
        % any irrational real numbers $\theta_1,\ldots,\theta_{k-1}$:\\
        % Let $J_j$ be compact intervals of positive length, $0\le j\le k-1$. %Further let $0<c<1/2$ be a constant.
        Let $k\ge 2$ be an integer and $J_0,\ldots,J_{k-1}$ be compact intervals of positive length. 
        For any $c_1>0$
        there exist $c>0$ such that
        for any large enough (depending on $J_j$) integer $\tilde{q}$ and any positive integers $R_1,\ldots,R_{k-1}$ with $(R_1R_2\cdots R_{k-1},\tilde{q})=1$, there exist integers $r_i,\tilde{r}_i$
        for $0\le i\le k-1$ and integers $s_j$ for $1\le j\le k-1$
        %with $(\tilde{q},\tilde{r}_j)=1$ for $0\le j\le k-1$
        such that
        \begin{equation}  \label{eq:c10}  \tag{c1a}
     R_j=-\tilde{q} r_j + s_j\tilde{r}_j, \quad 1\le j\le k-1 
        \end{equation}
        and
          \begin{equation}  \label{eq:c20}  \tag{c2a}
        \frac{  \tilde{r}_i }{ \tilde{q} }\in J_i, \quad 0\le i\le k-1
        \end{equation}
        and 
          \begin{equation}  \label{eq:c30}  \tag{c3a} 
          c_1< \frac{s_j}{ \tilde{q} }<c_1 + 1, \quad 1\le j\le k-1
        \end{equation}
        and 
          \begin{equation}  \label{eq:c40}  \tag{c4a} 
     \left\Vert \frac{ s_{j} \tilde{r}_{i} }{ \tilde{q} } - \theta_i \right\Vert \ge c, \qquad 1\le j\le k-1, \quad 0\le i\le k-1, \quad i\ne j.
        \end{equation}
        %Then Conjecture~\ref{konsch} holds for $k$.
    \end{lemma}

     \begin{proof}
        The first three conditions \eqref{eq:c10}-\eqref{eq:c30} are handled exactly same way as in Lemma~\ref{lemur0}. For last condition \eqref{eq:c40},
          since again $s_j\equiv R_j \tilde{r}_j^{-1}\bmod \tilde{q}$
          for $1\le j\le k-1$,
        we have to show
        \begin{equation} \label{eq:subj}
        R_j\tilde{r}_j^{-1}\cdot \tilde{r}_i\bmod \tilde{q}, \qquad 1\le j\le k-1, \; 0\le i\le k-1, \quad i\ne j
        \end{equation}
        can all avoid certain small intervals $G_{i,j}$ in $[0,\tilde{q}]\cap \mathbb{Z}$ of size $(1-2c)\tilde{q}$
        centered around $\lfloor\theta_i \tilde{q}\rfloor\bmod \tilde{q}$, under
        condition \eqref{eq:c20}.

       Now for each ordered pair $(i,j)$, $i\ne j$, 
       subject to the condition \eqref{eq:subj},
       again by~\cite[Theorem~13]{sh} we get essentially $(1-2c) \cdot \sharp (J_j \tilde{q}\cap \mathbb{Z})\cdot (\varphi(\tilde{q})/\tilde{q})$ many $\tilde{r}_j\in J_j \tilde{q} \cap \mathbb{Z}$ so that
       \eqref{eq:subj} holds for given $\tilde{r}_i$, in other words it holds for the $\kappa:= 1-2c>0$ ratio of relevant primitive integers in the interval $J_j$.
       Hence we conclude by the following combinatorial principle

       \begin{proposition}
           Let $A_1,\ldots,A_k$ be finite sufficiently large sets and 
           \[
           \kappa > 1- \frac{1}{2\binom{k}{2}}.
           \]
           Suppose for each ordered pair $(i,j)$ with $i\ne j$ we are given a set $\Omega_{i,j}\subseteq A_i\times A_j$ so that
           for every $a_i\in A_i$ the projection $E_{i,a_i}(j):=\{ a_j\in A_j: (a_i,a_j) \in \Omega_{i,j}\}\subseteq A_j$ has
           cardinality $\sharp E_{i,a_i}(j)\ge \kappa \sharp A_j$. 
           Then there exists
           $(a_1,\ldots,a_k)\in A_1\times A_2\cdots\times A_k$
           such that $(a_i,a_j)\in \Omega_{i,j}$ for each ordered pair $(i,j)$ with $i\ne j$.
       \end{proposition}

       \begin{proof}
           Fix an ordered pair $(i,j)$, $i\ne j$. There are $<\sharp A_i \cdot ((1-\kappa) \sharp A_j)$ many
           pairs $(a_i,a_j)\in (A_i \times A_j)\setminus \Omega_{i,j}$. Thus for
           at most $(1-\kappa)\prod_{g=1}^{k} \sharp A_g$ many tuples
           $(a_1,\ldots,a_k)\in A_1\times \cdots\times A_k$ this is the case.
           Thus if $2\binom{k}{2} \cdot (1-\kappa) < 1$ then there must remain
           some tuple for which $(a_i,a_j)\in \Omega_{i,j}$ holds for all ordered pairs $(i,j)$, $i\ne j$.
       \end{proof}

       The lemma is proved.
       \end{proof}
    We further need
    the following extension of Proposition 1, proved analogously.

  \begin{proposition} \label{pro00}
    Let $c,d$ positive real numbers with $cd<1$.
        Let $\xi$ be a real number and assume for some integers $p$ and $q>0$ we have
        \[
        |q\xi-p| < cq^{-k}.
        \]
        Then for any integers $u,v$ so that $u/v\ne p/q$ and 
        $0<v<Q=dq^{k}$, we have
        \[
        |v\xi-u| \ge (1-cd)d^{1/k} Q^{-1/k}.
        \]
    \end{proposition}

   Now the proof of Theorem~\ref{lemur00} works as follows: Start with $I_0$ and $J_{0,1},\ldots, J_{0,k-1}$
    arbitrary compact intervals. Assume the according 
    conditions \eqref{eq:2null}-\eqref{eq:EQA} from \S~\ref{casec} are again
    satisfied with respect to $\xi\in I_0$ and all $\zeta_j\in J_{0,j}$
    and $r_0, s_0, \tilde{r}_0$ replaced by $r_{0,j}, s_{0,j}, \tilde{r}_{0,j}$ (however with fixed $\tilde{q}_0$) for $1\le j\le k-1$, where
    in \eqref{eq:11null} we alter the exponent to $k+1$.
    Let accordingly 
    \[
    Q=Q_0=d\tilde{q}_0^{k}.
    \]
    Very similarly as in the proof of 
    we can handle the case where $(u,v_1,\ldots,v_{k})\in\mathbb{Z}^{k+1}$ is not 
    of the forms essentially as in \eqref{eq:1010}, which are now 
    a set of $k-1$ conditions, or \eqref{eq:2020}. Indeed,
    proceeding as in \S~\ref{casec} via Proposition~\ref{pro00} yields
    a lower bound
    \[
    Q|u\xi-v_1|\cdot \prod_{j=1}^{k-1} |u\zeta_j-v_{j+1}-\theta_j|\ge [(1-d\beta) d^{1/k}]^{k-1} \sigma \ge (1-d\beta)^{k} d/2>0.
    \]
    We are left with the
    cases  \eqref{eq:1010} and \eqref{eq:2020}.
    In the latter Case 1 of \eqref{eq:2020}, again
     \[
    \vert u\xi - v_1|\ge t |\tilde{q}_0 \xi - \tilde{p}_0|\ge |\tilde{q}_0 \xi - \tilde{p}_0| \ge \alpha d Q^{-1},
    \]
    and
     \[
    \vert u\zeta_j - v_{j+1}-\theta_j|\ge 
    \Vert\theta_j\Vert - \Vert u \zeta_j\Vert
    \ge (1-o(1)) \Vert\theta_j\Vert \gg_{\theta_j} 1, \qquad 1\le j\le k-1.
    \]
    Thus the product satisfies 
    \[
    Q|u\xi-v_1|\cdot \prod_{j=1}^{k-1} |u\zeta_j-v_{j+1}-\theta_j|
    \ge  (1-o(1)) \cdot \alpha d \prod_{j=1}^{k-1} \Vert\theta_j\Vert \gg_{\theta_1,\ldots,\theta_{k-1}} 1
    \]
    indeed.
    
    In former Case 2, we again can estimate
    \[
     |u\zeta_j-v_{j+1}-\theta_j| 
     =|s_{0,j} \zeta_j-r_{0,j}-\theta_j| + O(Q^{-(N-1)/2})
     > (c_1\alpha-o(1))d\cdot Q^{-1}, \quad 1\le j\le k-1,
     \]
     by the new analogous version of \eqref{eq:Gg}. 

     We take $\alpha, \beta, c_1, c_2:= c_1+1$ such that
     \[
     \alpha c_2 < \gamma_1 < \gamma_2 < \beta c_1,
     \]
     which is again possible by taking $\alpha$ small enough and $c_1$ large enough.
     Again proceeding as in~\S~\ref{casec} we require condition \eqref{eq:EQA} as before
     for $\xi$, and for $\zeta_i, i\ne j$ since again $\Vert t\tilde{q}_0 \zeta_j\Vert\ll \tilde{q}_0^{-1}$ is small,
     additionally conditions
     \[
     \left\Vert \frac{ s_{0,j} \tilde{r}_{0,i} }{ \tilde{q}_0 } - \theta_i \right\Vert \gg 1, \qquad 1\le i,j\le k-1, \; i\ne j.
     \]
     Here we have used assumption \eqref{eq:Diese}, note also the change of exponents to $k+1$ in \eqref{eq:11null}.
     So if we let $\theta_0=0$ and $\tilde{p}_0=\tilde{r}_{0,0}$ we can unify this as
     \[
     \left\Vert \frac{ s_{0,j} \tilde{r}_{0,i} }{ \tilde{q}_0 } - \theta_i \right\Vert \gg 1, \qquad 1\le j\le k-1, \quad 0\le i\le k-1, \quad i\ne j.
     \]
     This condition replaces \eqref{eq:c4} in Lemma~\ref{lemur0} which 
     accordingly with $I=:J_0$ becomes \eqref{eq:c40} in Lemma~\ref{lemur00}.
     Existence of the desired real vectors follows now again analogously to \S~\ref{casec}, the density and $G_{\delta}$ properties follow analogously as well.

\end{document}